\documentclass[a4paper,12pt]{ETHpaper}
	
	\usepackage{graphicx}\usepackage[dvipsnames]{xcolor}\usepackage{dcolumn}\usepackage{bm}
	\usepackage[british]{babel}
	\usepackage{subfigure}
	\usepackage{amssymb,amsmath,amsthm}
	\usepackage[capitalise]{cleveref}
	
	\newcommand{\N}{\mathbb{N}}
	\newcommand{\R}{\mathbb{R}}
	\newcommand{\G}{\mathcal{G}}
	\newcommand{\expected}[1]{\mathbb{E}\left[ #1 \right]}
	 \newcommand{\abs}[1]{\left| #1 \right|}
	
	\newcommand{\kseq}{\mathbf{k}}
	\newcommand{\dir}[1]{#1^{\scriptscriptstyle\updownarrow}}

	\theoremstyle{definition}
	\newtheorem{defn}{Definition}
	
	\theoremstyle{plain}
	\newtheorem{thm}{Theorem}
	\newtheorem{lemma}[thm]{Lemma}
	\newtheorem{cor}{Corollary}[thm]
	\crefname{cor}{Corollary}{Corollaries}

\begin{document}

\title{Generalised hypergeometric ensembles of random graphs:\\
the configuration model as an urn problem}
\titlealternative{Generalised hypergeometric ensembles of random graphs}

\author{Giona Casiraghi\(^*\) \qquad Vahan Nanumyan\(^\dagger\)}
\authoralternative{G. Casiraghi, V. Nanumyan}
\address{Chair of Systems Design, \\ ETH Zurich, 
Weinbergstrasse 56/58, 8092 Zurich, Switzerland \\[2mm]
\(^*\)gcasiraghi@ethz.ch\\
\(^\dagger\)vnanumyan@ethz.ch
}
\www{\url{http://www.sg.ethz.ch}}
\maketitle

\begin{abstract}
	We introduce a broad class of random graph models: the generalised hypergeometric ensemble (GHypEG).
	This class enables to solve some long standing problems in random graph theory.
	First, GHypEG provides an elegant and compact formulation of the well-known configuration model in terms of an urn problem.
	Second, GHypEG allows to incorporate arbitrary tendencies to connect different vertex pairs.
	Third, we present the closed-form expressions of the associated probability distribution ensures the analytical tractability of our formulation.
	This is in stark contrast with the previous state-of-the-art, which is to implement the configuration model by means of computationally expensive procedures.
	
	\vspace{1em}
\end{abstract}

\section{Introduction}

Important features of real-world graphs are often analysed by studying the deviations of empirical observations from suitable random models.
Such models, or graph ensembles, are built so that some of the properties of the analysed empirical graph are preserved.
Then, one can identify which other features of the empirical graph can be expected at random and which cannot, given the encoded constraints.

The simplest random graph model, named after Erd\"os and R\'enyi, generates edges between a given number of vertices with a fixed probability \(p \in (0,1]\)~\cite{erdds1959random}.
In this model, the properties of vertices, such as their degrees, have all the same expected value.
However, most empirical graphs have heterogeneous, heavy tailed degree distributions~\cite{Aiello2000,Huberman1999,Chung2002a,barabasi1999emergence}.
Hence, random graph models able to incorporate arbitrary degree sequences are of special importance.

The most common model fixing only degree sequences is known as the \emph{configuration model} of random graphs~\cite{Chung2002a,Chung2002,Bender1978,Molloy1995}.
The comparison of an empirical graph with the corresponding configuration model then allows to quantify which properties of the original graph can be ascribed to the degree sequence.
The properties not explained by the degree sequence highlight the unique structure of the studied empirical graph.

The standard configuration model has a crucial drawback: the lack of analytical tractability.
In fact, the model is realised by means of a repeated rewiring procedure.
Each vertex is assigned a number of half-edges, or stubs, corresponding to their degree and one random realisation of the model is obtained by wiring pairs of stubs together uniformly at random.
This is a computationally expensive procedure, which does not allow to explore the whole probability space of the model.
This problem is exacerbated in the case of larger graphs and graphs with highly heterogeneous degree distributions.

Moreover, the standard configuration model is limited to the coarse analysis of the combinatorial randomness arising from vertex degrees, as it has no free parameters.
While it has been invaluable for the macroscopic and the mesoscopic analysis of graphs, such as for graph partitioning through modularity maximisation~\cite{Newman2006} and for quantifying degree correlations~\cite{Newman2003}, other graph models are needed to address complex dyadic patterns beyond degrees.

In this article, we propose a novel analytically tractable model for random graphs with given expected degree sequences.
The formulation of the model relies on mapping the process of drawing edges to a multivariate urn problem.
In its simplest case, our model corresponds to the configuration model for directed or undirected multi-edge graphs, fixing in expectation the values of vertex degrees instead of their exact values.
In the general case, the model incorporates a parameter for each pair of vertices, which we call \emph{edge propensity}.
This parameter controls the relative likelihood of drawing an individual edge between the respective pair of vertices, as opposed to any other pair.
This is achieved by biasing the combinatorial edge drawing process.
The proposed formulation allows to model and test for arbitrary graph patterns that can be reduced to dyadic relations.

\section{Results}

Let us consider a \emph{multi-graph} \(\G=(V,E)\), where \(V\) is a set of \(n\) vertices, and \(E \subseteq V \times V\) is a multi-set of \(m\) (directed or undirected) multi-edges.
For clarity, we first provide working definitions of multi-edges and multi-graphs.
\begin{defn}[Multi-edges]
	Let \(V\subset\N\) be a set of vertices and \(i,j\in V\) two vertices.
	Elements \((i,j)_l \in E\) and \((i,j)_{k} \in E\), \(l\neq k\), incident to the same two vertices are called multi-edges.
	The number \(A_{ij}\in\N_0\)  of multi-edges incident to the same two vertices \(i\) and \(j\) defines the \emph{multiplicity} of the edge  \((i,j)\).
\end{defn}

\begin{defn}[Multi-graph]
	Let \(V\subset\N\) be a set of vertices.
	A graph \(\mathcal G(V,E)\) is called a multi-graph if \(E\subseteq V\times V\) is a \emph{multi-set} of \(m :=\abs{E}\) multi-edges.
	Self-loops \((i,i)\in E\) for \(i\in V\) are generally allowed.
	A multi-graph \(\G\) can be directed or undirected.
\end{defn}

We indicate with \(\mathbf A\) the adjacency matrix of the graph where entries \(A_{ij}\in \N_{0}\) capture the multiplicity of an edge \((i,j)\in V \times V\) in the multi-set \(E\).
In the case of undirected graphs, the adjacency matrix is symmetric, i.e. \(\mathbf{A} = \mathbf{A}^T \), and the elements on its diagonal equal twice the multiplicity of the corresponding self-loops.

\begin{defn}[Degrees]
	For each vertex \(i \in V\) the in-degree \(k^{\mathrm{in}}_i := \sum_{j \in V} A_{ji}\) and the out-degree \(k^{\mathrm{out}}_i := \sum_{j \in V} A_{ij}\).
	The total number of multi-edges is expressed as \(m=\sum_{i,j \in V}A_{ij}=\sum_{i \in V}k^{\mathrm{out}}_i = \sum_{j \in V}k^{\mathrm{in}}_j\).
	We denote the in-degree and out-degree \emph{sequence} of a directed graph \(\G\) as \(\kseq^\mathrm{in}(\G)=\{k^{\mathrm{in}}_i\}_{i\in V}\) and \(\kseq^\mathrm{out}(\G)=\{k^{\mathrm{out}}_i\}_{i\in V}\).
	We denote with \(k^\mathrm{in/out}_i(\G)\) the \(i\)-th entry of the degree sequence \(\kseq^\mathrm{in/out}(\G)\), corresponding to the  in- or out-degree of vertex \(i\).
	For undirected graphs, the adjacency matrix is symmetric and thus \(k^{\mathrm{in}}_i=k^{\mathrm{out}}_i=: k_i\). Hence, there is one degree sequence of an undirected graph, which we denote \(\kseq(\G)\).
	
\end{defn}

As we only deal with multi-graphs,  we will refer to multi-graphs simply as graphs  in the rest of the article.

\subsection{Soft Configuration Model}\label{sec:conf}

The concept underlying our random graph model is the same as for the standard configuration model of Molloy and Reed~\cite{Molloy1995,molloy_reed_1998}, which is to randomly shuffle the multi-edges of a graph \(\G\) while preserving vertex degrees.
The standard configuration model generates multi-edges one after another by sampling uniformly at random a vertex with an available out-stub (outwards half-edge) and a vertex with an available in-stub (inwards half-edge), until all stubs are consumed. 
\Cref{fig:softconf} illustrates one step of this process.
The resulting random graphs all have exactly the same degree sequence as the original graph \(\G\).
All pairs of available in- and out-stubs are equiprobable to be picked, so are the corresponding individual multi-edges.
Therefore, the probability to observe a multi-edge between a given pair of vertices positively relates to the number of possible stub pairings of the two vertices, which in turn is defined by the corresponding degrees of these.
In fact, this probability depends only on the degrees of the two vertices and on the total number of multi-edges in the graph.
\begin{figure}[tph]
	\centering
	\includegraphics[width=.85\textwidth]{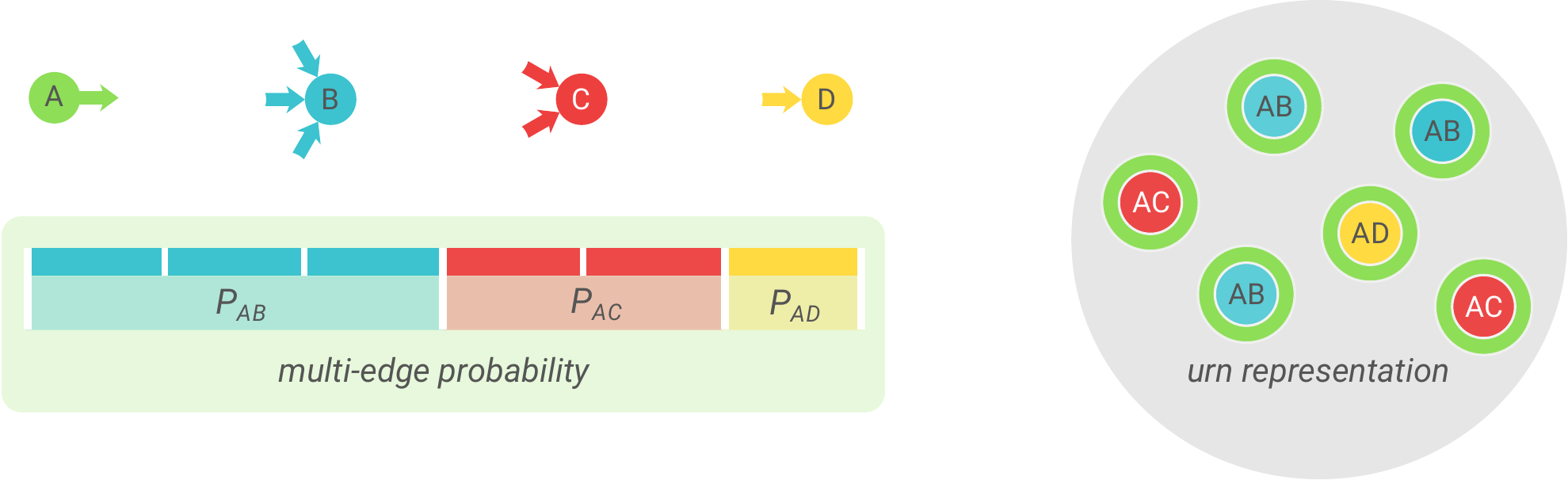}
	\caption{
	The configuration model represented (\textbf{upper left}) as a conventional edge rewiring process and (\textbf{right}) as an urn problem.
	In the former case, once the out-stub \((A,\cdot)\) has been sampled for rewiring, then one in-stub is sampled uniformly at random from those available, to draw a new multi-edge.
	If we represent each possible combination of an out-stub and an in-stub as a ball, we arrive at the urn problem without replacement.
	For the shown vertices, the odds of observing a multi-edge \((A,B)\) are three times higher than of observing a multi-edge between \((A,D)\) and \(1.5\) times higher than of observing a multi-edge between \((A,C)\) in both model representations.
	}\label{fig:softconf}
\end{figure}

The need to consequently sample two vertices at each step makes it cumbersome to analytically formulate the procedure described above.
To overcome this challenge, we take an \emph{edge-centric approach} of sampling \(m\) multi-edges from a certain larger multi-set, which we define below in \cref{def:xi}.
As a consequence of this change of perspective the model will preserve the \emph{expected} degree sequences instead of the exact ones.
To this end, we introduce the definition of the \emph{soft configuration model}.

For each pair of vertices \(i,j\in V\), we define the number \(\Xi_{ij}\) of stub combinations that exist between vertices \(i\) and \(j\), which can be conveniently represented in matrix form:

\begin{defn}[Combinatorial matrix] \label{def:xi}
	We define combinatorial matrix \(\mathbf\Xi\in\N^n\times\N^n\) for graph \(\G\) as the matrix whose elements \(\Xi_{ij}\) are defined as
	\begin{equation}
		\label{eq:xi}
		\Xi_{ij} = k^{\mathrm{out}}_i(\G) k^{\mathrm{in}}_j(\G) \quad \mathrm{for~} i,j \in V,
	\end{equation}
	where  \(\kseq^{\mathrm{out}}(\G)\) and \(\kseq^{\mathrm{in}}(\G)\) are the out-degree and in-degree sequences of the graph \(\G\).
	\end{defn}

\subsubsection{Soft configuration model for directed graphs}

\begin{defn}[Directed soft configuration model]\label{def:softconf}
	Let \(\hat \kseq^{\mathrm{in}},\,\hat \kseq^{\mathrm{out}}\in\N^n\) be in- and out-degree sequences and \(\hat V\) a set of \(n\) vertices.
	The soft configuration model \(X\) generated by \((\hat V,\hat \kseq^{\mathrm{in}},\,\hat \kseq^{\mathrm{out}})\) is the \(n^2\)-dimensional random vector \(X\) defined on the probability space \((S,P)\) with sample space
	\begin{equation}
		\label{eq:softconf}
		S=\left\{ \G(V,E) \,\big\vert\, \abs{E} = m \right\},\quad m = \sum_{i\in V}{\hat \kseq^{\mathrm{in}}_i}=\sum_{i\in V}{\hat \kseq^{\mathrm{out}}_i},
	\end{equation}
	with some probability measure \(P\), such that the expected degree sequences of a realisation of \(X\) are fixed:
	\begin{equation}
		\mathbb E_P\left[  \kseq^{\mathrm{in}}(X)\right]=\hat \kseq^{\mathrm{in}},\quad \mathbb E_P\left[  \kseq^{\mathrm{out}}(X)\right]=\hat \kseq^{\mathrm{out}}.
	\end{equation}
\end{defn}

This set-up allows to map the model to an \emph{urn problem} and thus to arrive to a closed-form probability distribution function for it.

\begin{lemma}[Number of stub combinations]
	\label{lem:xidir}
	The combinatorial matrix \(\mathbf\Xi\in\N^n\times\N^n\) given in \cref{def:xi} encodes the numbers of out-stub and in-stub combinations for each pair of vertices, given degree sequences \(\kseq^{\mathrm{out}}\) and \(\kseq^{\mathrm{in}}\).
	\end{lemma}
\begin{proof}
	Let \(\kseq^{\mathrm{out}}_i\) be the out-degree of vertex \(i\) and \(\kseq^{\mathrm{in}}_j\) the in-degree of vertex \(j\).
	The number of out-stubs of a vertex corresponds to its out-degree.
	Similarly, the number of in-stubs of a vertex corresponds to its in-degree.
	Each one of the \(\kseq^{\mathrm{out}}_i\) out-stubs can be connected to all \(\kseq^{\mathrm{in}}_j\) in-stubs.
	Hence, the total number of stub combinations between vertices \(i\) and \(j\) \(\Xi_{ij}\) is equal to \(\kseq^{\mathrm{out}}_i \kseq^{\mathrm{in}}_j\).
\end{proof}

We further introduce the concept of \emph{induced random model}.
\begin{defn}[Graph-induced random model]
	We say that the graph \(\G(V,E)\) induces the random model \(X\), if the quantities \((\hat V,\hat \kseq^{\mathrm{in}},\hat \kseq^{\mathrm{out}})\) generating \(X\) are computed from \(\G\).
	I.e., \(\hat V=V\), \(\hat \kseq^{\mathrm{in}}=\kseq^{\mathrm{in}}(\G)\), and \(\hat \kseq^{\mathrm{out}}=\kseq^{\mathrm{out}}(\G)\).
\end{defn}

Under this assumption, we can formulate the following theorem for the distribution of the soft configuration model \(X\).
To keep the notation simple, we will not distinguish between the \(n\times n\) adjacency matrix \(\mathbf{A}\) and the vector of length \(n^{2}\) obtained by stacking it by row or column.
Similarly, we do the same for all other related \(n\times n\)  matrices.
\begin{thm}\label{thm:softconf}
	Let \(\mathcal G(V,E)\) be a directed graph with \(n=\abs{V}\) vertices and \(m=\abs{E}\) multi-edges.
	Let \(\kseq^{\mathrm{in}}(\G)\in\N^n\) and \(\kseq^{\mathrm{out}}(\G)\in\N^n\) be the vectors representing its in-degree and out-degree sequences.
	Let X be the soft configuration model induced by \(\mathcal G\) defined as in \cref{def:softconf}.
	If the probability measure \(P\) depends only on the degree sequences in \(\G\) and the total number of multi-edges \(m=\abs{E}\) and all multi-edges are equiprobable, then \(X\) follows the multivariate hypergeometric distribution as in \cref{eq:hyper}.

	Let \(\mathbf A\in \N_0^n\times\N_0^n\) be an adjacency matrix and \(\mathbf\Xi\in \N^n\times\N^n\) be the \emph{combinatorial matrix} induced by \(\G\).
	Then the soft configuration model \(X\) is distributed as follows:
	\begin{equation}\label{eq:hyper}
		\Pr(X=\G) = \frac{\prod_{i,j\in V}\dbinom{\Xi_{ij}}{A_{ij}}}{\dbinom{M}{m}},
	\end{equation}
	where \(M=\sum_{i,j\in V}\Xi_{ij}\) is the total number of stub combinations between all vertices.
\end{thm}
\begin{proof}
	
	We want to sample \(m\) multi-edges connecting any of the in- and out-stub pairs such that all such such multi-edges are equiprobable.
	According to \cref{lem:xidir}, the total number of stubs combinations between any two vertices \(i,j\) is given by \(\Xi_{ij}\).
	We can hence define the random graph model as follows.
	We sample \(m\) multi-edges without replacement from the multi-set of size \(\sum_{i,j\in V}\Xi_{ij}\) that combines all the possible stub pairs combinations \(\Xi_{ij}\) between all pairs \(i,j \in V\).
	We sample without replacement because we need to mimic the process of wiring stubs.
	Once a stub pair has been used it cannot be sampled again.
	We can view this model as an \emph{urn problem} where the edges to be sampled are represented by balls in an urn.
	By representing the multi-edges connecting each different pair of vertices \((i,j)\) as balls of a unique colour, we obtain an urn with a total of \(M=\sum_{i,j\in V}\Xi_{ij}\) balls of \(n^{2}=\abs{V\times V}\) different colours.
	With this, the sampling of a graph according to our model corresponds to drawing exactly \(m\) balls from this urn.
	Each adjacency matrix \(\mathbf{A}\) with \(\sum_{i,j\in V} A_{ij}=m\) corresponds to one particular realisation drawn from this model.
	The probability to draw exactly \(\mathbf{A}=\{A_{ij}\}_{i,j \in  V}\) edges between each pair of vertices is given by the multivariate hypergeometric distribution.
\end{proof}

From \cref{thm:softconf} we derive the following results, whose proofs follow directly from properties of the hypergeometric distribution.
\begin{cor}\label{cor:margdir}
	For each pair of vertices \(i,j \in V\), the probability that \(X\) has exactly \({A}_{ij}\) edges between \(i\) and \(j\) is given by the marginal distributions of the multivariate hypergeometric distribution, i.e.
	\begin{equation}
		\label{eq:hypergeometricEdge}
		\Pr(X_{ij}={A}_{ij}) = \frac{\dbinom{\Xi_{ij}}{{A}_{ij}}\dbinom{M-\Xi_{ij}}{m-{A}_{ij}}}{\dbinom{M}{m}}.
	\end{equation}
	\end{cor}

\begin{cor}\label{cor:expsoftdir}
	The \emph{expected} in- and out-degree sequences of realisations of the directed soft configuration model \(X\) correspond to the respective degree sequences of the graph \(\G\) inducing \(X\).
\end{cor}
\begin{proof}
	For each pair of vertices \(i,j\) we can calculate the expected number of multi-edges \(\expected{X_{ij}}\) as
	\begin{equation}
		\expected{X_{ij}} = m \frac{\Xi_{ij}}{M}
	\end{equation}
	Moreover, summing the rows and columns of matrix \(\expected{X_{ij}}\) and assuming directed graphs with self-loops we can calculate the expected in- or out-degrees of all vertices as
	\begin{equation}
		\label{eq:hypergeometricDegree}
		\begin{aligned}
			\expected{k^{\mathrm{in}}_j(X)} =& \sum_{i \in V}{\expected{X_{ij}}} = m \frac{\sum_{i \in V} \hat{k}^{\mathrm{out}}_i \hat{k}^{\mathrm{in}}_j}{M} = \hat{k}^{\mathrm{in}}_j,  \\
			\expected{k^{\mathrm{out}}_i(X)} =& \sum_{j \in V}{\expected{X_{ij}}} = m \frac{\sum_{j \in V} \hat{k}^{\mathrm{out}}_i \hat{k}^{\mathrm{in}}_j}{M} = \hat{k}^{\mathrm{out}}_i.
		\end{aligned}
	\end{equation}

	\Cref{eq:hypergeometricDegree} confirms that the \emph{expected}  in- and out-degree sequence of realisations drawn from \(X\) corresponds to the degree sequence of the given graph \(\G\).
\end{proof}

\subsubsection{Soft configuration model for undirected graphs}

So far we have discussed the configuration model for directed graphs.
Specifying the undirected case, on the other hand, requires some more efforts.
The reason for this is that, under the assumptions described in the previous section, the undirected version of the soft configuration model is the degenerate case of its directed counterpart, where the direction of the multi-edges is ignored.
In particular, this implies that the random vector corresponding to the undirected model has half the dimensions of the directed one, because any undirected multi-edge between two vertices \(i\) and \(j\) can either be generated as a directed multi-edge \((i,j)\) or as a directed multi-edge \((j,i)\).

\begin{defn}[Undirected soft configuration model]\label{def:softconf-undir}
	Let \(\hat \kseq\in\N^n\) be a degree sequence and \(\hat V\) a set of \(n\) vertices.
	The soft configuration model \(X\) generated by \((\hat V,\hat \kseq)\) is the \((n^2+n)/2\)-dimensional random vector \(X\) defined on the probability space \((S,P)\), for the sample space
	\begin{equation}
		\label{eq:softconf-undir}
		S=\left\{ \G(V,E) \,\big\vert\, \abs{E} = m \right\},\quad 2m = \sum_{i\in V}{\hat \kseq_i},
	\end{equation}
	with some probability measure \(P\), such that the expected degree sequence of a realisation of \(X\) is fixed:
	\begin{equation}
		\mathbb E_P\left[  \kseq(X)\right]=\hat \kseq.
	\end{equation}
\end{defn}

At first sight, it would appear that the undirected soft configuration model can simply be obtained by restricting the directed model to \(n(n+1)/2\) components corresponding to the upper-triangle and the diagonal of the adjacency matrix.
That is, we would sample \(m\) multi-edges among pairs \(i\leq j \in V\) from the multi-set of stub combinations \(\Xi_{ij}\) as defined in \cref{eq:xi}.
However, the resulting model does not satisfy \cref{def:softconf-undir}, because its expected degree sequence does not equal the degree sequence that induced it.
To show this, we follow the same reasoning adopted in the proof of \cref{cor:expsoftdir}.
The expected degree \(\expected{k_i(X)}\) of a vertex \(i\) is equivalent to
\[
\expected{k_i(X)} = \sum_{j \in V}{\expected{X_{ij}}} = m \frac{\sum_{i \in V} \hat{k}_i \hat{k}_j}{\sum_{i\leq j\in V}\hat{k}_i \hat{k}_j} \neq \hat{k}_i.
\]

This approach is wrong because the total number of undirected stub combinations is larger than in the directed case. The reason for this is the symmetry in the process of wiring two stubs.
Let \(\hat \kseq_i\) be the degree of vertex \(i\) and \(\hat \kseq_j\) the degree of vertex \(j\).
To form a multi-edge \((i,j)\), each one of the \(\hat \kseq_i\) stubs of \(i\) can be connected to all \(\hat \kseq_j\) stubs of \(j\), and vice versa, each of the \(\hat \kseq_j\) stubs of \(j\) can be connected to all \(\hat \kseq_i\) stubs of \(i\).
Hence, the total number of combinations of stubs between vertices \(i\) and \(j\) equals to \(\hat \kseq_i \hat \kseq_j + \hat \kseq_j \hat \kseq_i = 2 \hat \kseq_i \hat \kseq_j\).
As an equivalent to \cref{lem:xidir}, we formalise this in the following lemma.

\begin{lemma}[Undirected stubs combination count]
	\label{lem:xiundir}
	The total number of stub combinations between two vertices \(i,j\) in an undirected graph is given by:
	\begin{equation}\label{eq:combinations-undir}
		\begin{cases}
			2 \Xi_{ij} & \mathrm{if~} i\neq j, \\
			\Xi_{ii} & \mathrm{if~} i=j, \\
		\end{cases}
	\end{equation}
\end{lemma}

We can now formulate the equivalent of \cref{thm:softconf} for the undirected case.
The distribution underlying the undirected soft configuration model can then be computed analogously to \cref{thm:softconf} with the help of \cref{lem:xiundir}.

\begin{thm}\label{thm:undirected}
	Let \(\G(V,E)\) be an undirected graph with \(n=\abs{V}\) vertices and \(m=\abs{E}\) multi-edges.
	Let \(\kseq \in \N^n\) be the vector representing its degree sequence.
	Let \(X\) be the undirected soft configuration model induced by \(\G\) defined as in \cref{def:softconf-undir}.
	If the probability distribution underlying \(X\) depends only on the degree sequence of \(\G\) and the total number of multi-edges \(m\), and all multi-edges are equiprobable, then \(X\) follows the multivariate hypergeometric distribution given in \cref{eq:hyper-undir}.
	
	Let \(\mathbf{A} \in \N_0^n \times \N_0^n\) be the symmetric adjacency matrix corresponding to an undirected graph, and \(\Xi_{ij} = \kseq_i \kseq_j\) be the combinatorial matrix induced by \(\G\).
	Then the undirected soft configuration model \(X\) is distributed as follows:
	\begin{equation}\label{eq:hyper-undir}
		\Pr(X=\G) = \frac{\prod_{i<j\in V}\dbinom{2\Xi_{ij}}{A_{ij}} \prod_{l\in V}\dbinom{\Xi_{ll}}{A_{ll}/2}}{\dbinom{M}{m}},
	\end{equation}
	where \(M=\sum_{i<j\in V} 2\Xi_{ij} + \sum_{l\in V} \Xi_{ll} = \sum_{i,j\in V} \Xi_{ij}\) is the total number of undirected stub combinations between all pair of vertices. 
\end{thm}
\begin{proof}
	The proof follows the same reasoning of the proof of \cref{thm:softconf}, accounting for the fact that the total number of stubs combinations is now given by \cref{eq:combinations-undir}.
\end{proof}

\begin{cor}\label{cor:margundir}
	For each pair of vertices \(i,j \in V\), the probability that \(X\) has exactly \({A}_{ij}\) edges between \(i\) and \(j\) is given by the marginal distributions of the multivariate hypergeometric distribution in \cref{eq:hyper-undir}, i.e.,
	\begin{equation}
		\label{eq:hypergeometricEdgeUndir}
		\Pr(X_{ij}={A}_{ij}) = 
		\begin{cases}
			\dbinom{2\Xi_{ij}}{{A}_{ij}}\dbinom{M-2\Xi_{ij}}{m-{A}_{ij}}{\dbinom{M}{m}}^{-1} & \mathrm{for~} i \neq j,\\[1.5em]
			\dbinom{\Xi_{ij}}{{A}_{ij}/2}\dbinom{M-\Xi_{ij}}{m-{A}_{ij}/2}{\dbinom{M}{m}}^{-1} & \mathrm{for~} i = j.
		\end{cases}
	\end{equation}
	\end{cor}

\begin{cor}\label{cor:expsoftundir}
	The \emph{expected} degree sequence of realisations of \(X\) correspond to the respective degree sequences of the graph \(\G\) inducing \(X\).
\end{cor}
\begin{proof}
	For each pair of vertices \(i,j\in V\), the expected number of multi-edges \(\expected{X_{ij}}\) according to the hypergeometric distribution in \cref{eq:hyper-undir} is expressed as
	\begin{equation}
		\expected{X_{ij}} = 2 m \frac{\Xi_{ij}}{M}
	\end{equation}
	With this, we can write the expected degrees as
	\begin{align}
		\expected{k_j(X)} 
		&= \sum_{i \in V} {\expected{X_{ij}}} = 2 m \frac{\sum_{i \in V}\Xi_{ij}}{M} = 2 m  \frac{\sum_{i \in V} \hat k_i \hat k_j}{\sum_{i,j \in V} \hat k_i \hat k_j} = \hat k_j.
	\end{align}
\end{proof}

\subsubsection{Correspondence between directed and undirected models}

In the previous two sections, we have formulated the soft configuration model for directed and undirected graphs independently of each other.
As the reader recalls, we have motivated these models by the need of an analytically tractable analogy for the rewiring algorithm of the Molloy-Reed model~\cite{Molloy1995}.
This algorithm is the same in the directed and undirected case: select the first stub (outgoing, in the directed case), then select the second stub (incoming, in the directed case), create an edge by wiring these two stubs, and repeat the process until all the stubs are wired.
Hence, we also show the correspondence between our directed and undirected formulations in this section.

We prove that the probability distribution of undirected graphs in the undirected soft configuration model given by \cref{eq:hyper-undir} is a degenerate case of the directed model given by \cref{eq:hyper}.

With the following definition we provide a projection from \(\N^{n^2}\) to \(\N^{n(n+1)/2}\) that serves the purpose mapping a directed graph to its undirected equivalent, i.e., stripping the direction from its edges.

\begin{defn}[Undirected projection]\label{def:undirectionalise}
	Let \(\dir{\G}(V,\dir{E})\) be a directed graph with adjacency matrix \(\dir{\mathbf{A}}\).
	We define as \emph{undirected projection} the map \(\pi:\N^{n^2}\to\N^{n(n+1)/2}\) that maps \(\dir\G\) to the undirected graph \(\G(V,E)\) with adjacency matrix \(\mathbf{A}=\mathbf{\dir{A}} + \mathbf{\dir{A}}^T\).
	We indicate with \(\dir\G\hookrightarrow\G\) the fact that \(\G=\pi(\dir\G)\).
\end{defn}

According to \cref{def:undirectionalise} there are different directed graphs that can be projected to the same undirected graph.
At the same time, every undirected graph has at least one corresponding directed graph that can be projected to it, and for every directed graph there is at least one undirected graph to which it can be projected.
These make the projection in \cref{def:undirectionalise} surjective and not injective.

Similarly, we can define an undirected random graph model as the projection of a directed random graph model.

\begin{defn}\label{def:undirproj} Let \(\dir X\) be a directed random graph model.
	With an abuse of notation, we use \({\dir X}^T\) to refer to the transposition of the matrix representation of \(\dir X\).
	We say that \(X:=\dir X + {\dir X}^T\) is the undirected projection of \(\dir X\) if \(\forall \dir\G\) in the sample space of \(\dir X\) exists a \(\G\) in the sample space of \(X\) such that the undirected projection \(\pi(\dir\G)\) of \(\dir\G\) is \(\G\).
	Furthermore, for every undirected graph \(\G\) in the sample space of \(X\), \(\exists \dir{\G}\) such that \(\pi(\dir\G) = \G\).
	We indicate with \(\dir X\hookrightarrow X\) the fact that \(X\) is the undirected projection of \(\dir X\).
\end{defn}

Note that according to \cref{def:undirectionalise}, the number of multi-edges \(m\) of \(\G\) equals the number of multi-edges of any directed \(\dir{\G}\) that projects to \(\G\).

Finally we need to relate the distribution underlying a directed random graph model to the the distribution of its undirected projection.
The following lemma serves this purpose.

\begin{lemma}\label{lem:projdist}
	Let \(\G\) be an undirected graph and \(X\) an undirected random graph model.
	Let \(\dir X\) be a directed random graph model such that \(\dir X\hookrightarrow X\).
	The probability distribution of \(X\), \(\Pr(X=\G)\), is given as:
	\begin{equation}\label{eq:undirmap}
		\Pr(X=\G)=\sum_{\dir\G\in \pi^{-1}(\G)}\Pr\left(\dir{X}=\dir{\G}\right),
	\end{equation}
	where the set \( \pi^{-1}(G)=\left\{ \dir\G\;\vert\dir\G\hookrightarrow\G \right\}\) is the set of all directed graphs \(\dir\G\) that map to \(\G\).
\end{lemma}
\begin{proof}
	Let \(\dir X\) be a \(n^2-\)dimensional random vector formalising a directed random graph model, such that \(\dir X\hookrightarrow X\).
	For simplicity we index the elements of both random vectors as in the equivalent adjacency matrix notation.
	Let \(X_{ij}=X_{ji}\) the ij-th element of \(X\) and \(\dir X_{ij}, \dir X_{ji}\) the corresponding elements of \(\dir X\).
	
	According to \cref{def:undirproj}, \(X\) is the \(n(n+1)/2\)--dimensional random vector defined as \(\dir X + {\dir X}^T\), where its each element \(ij\) is defined as \(X_{ij}=\dir X_{ij} + {\dir X}^T_{ji}\). The probability distribution of \(X\), \(\Pr(X=\G)=f_X(\G)\) can be specified in terms of the probability distribution \(f_{\dir{X}}(\dir\G)=\Pr(\dir X=\dir\G)\):	
	\begin{align}
		f_X(z) = f_X\left(\{z_{ij}\}_{ij}\right) = \sum\cdots\sum_{a_{ij}=0}^{z_{ij}} f_{\dir{X}}\left(\{z_{ij}-a_{ij},a_{ij}\}_{ij,ji}\right)\label{eq:prooflemma}
	\end{align}
	The summation in \cref{eq:prooflemma} corresponds to the sum over the probabilities of all possible combinations of tuples \(\dir X_{ij},\dir X_{ji}\) which sum to \(A_{ij}\) for all indices \(ij\).
	Hence, following \cref{def:undirectionalise}, this is equivalent to sum over all possible \(\dir\G\hookrightarrow\G\).
	This proves the equivalence between \cref{eq:prooflemma} and \cref{eq:undirmap} and thus, the lemma.
\end{proof}

We can proceed showing that the undirected version of the soft-configuration model given in \cref{thm:softconf} is indeed equivalent to the model defined in \cref{thm:undirected}.
\Cref{thm:equivalence} stems from the fact that sampling an undirected edge between two vertices is equivalent to sampling a directed edge between the same pair of vertices in any of the two directions, and then stripping its direction information.
The distribution underlying the undirected soft configuration model can then be computed with the help of \cref{lem:projdist}.

\begin{thm}\label{thm:equivalence}
	Let \(\G\) be an undirected graph and \(X\) the undirected soft-configuration model.
	Let \(\dir X\) be the directed soft configuration model with combinatorial matrix with elements \(\Xi_{ij} = \kseq_i \kseq_j\).
	The probability distribution of \(X\) is then given by \cref{eq:hyper-undir}.
\end{thm}
\begin{proof}
	The distribution of \(\dir X\) is given by the hypergeometric distribution in \cref{eq:hyper}.
	The model \(\dir X\) satisfies the conditions in \cref{lem:projdist} for the undirected soft configuration model \(X\), because \(\dir X\) maps to the undirected soft configuration model \(X\) in accordance with \cref{def:undirproj}.
	Hence, we write the probability distribution underlying \(X\) as the sum of the probabilities of all corresponding directed graphs \(\dir{\G}\) under the directed soft configuration model \(\dir X\).
	\begin{align}
		\Pr(X=\G)&=\sum_{\dir\G\in \pi^{-1}(\G)}\Pr\left(\dir{X}=\dir{\G}\right)\\
		&= \sum_{\dir\G\in \pi^{-1}(\G)}\dbinom{M}{m}^{-1} \prod_{i,j\in V}\dbinom{\Xi_{ij}}{\dir{A}_{ij}}
		\label{eq:sum-dir:2} \\
		&= \sum_{\dir\G\in \pi^{-1}(\G)}\dbinom{M}{m}^{-1} \prod_{l\in V}\dbinom{\Xi_{ll}}{\dir{A}_{ll}} \prod_{i<j\in V}\dbinom{\Xi_{ij}}{\dir{A}_{ij}} \dbinom{\Xi_{ij}}{{A}_{ij} - \dir{A}_{ij}}
		\label{eq:sum-dir:3} \\
		&= \sum \dots \sum_{\dir{A}_{ij}=0}^{A_{ij}} \dbinom{M}{m}^{-1} \prod_{l\in V}\dbinom{\Xi_{ll}}{\dir{A}_{ll}} \prod_{i<j\in V}\dbinom{\Xi_{ij}}{\dir{A}_{ij}} \dbinom{\Xi_{ij}}{{A}_{ij} - \dir{A}_{ij}}
		\label{eq:sum-dir:4}
		\end{align}
	In \cref{eq:sum-dir:4} we have \(n(n-1)/2\) summations for all \(\dir{A}_{ij}\), \(i < j\), which are the decomposition of the summation in \cref{eq:sum-dir:3}.
	Then, we can swap the summations and multiplications in \cref{eq:sum-dir:4}, which leads to 
	\begin{equation}
		\Pr(\mathbf{A}) = 
		\dbinom{M}{m}^{-1} \prod_{l\in V}\dbinom{\Xi_{ll}}{\dir{A}_{ll}} \prod_{i<j\in V}\sum_{\dir{A}_{ij} = 0}^{A_{ij}} \dbinom{\Xi_{ij}}{\dir{A}_{ij}} \dbinom{\Xi_{ij}}{{A}_{ij} - \dir{A}_{ij}}.
		\label{eq:sum-dir:fold}
	\end{equation}
	From Vandermonde's identity, which states
	\begin{equation}
		\sum_{a=0}^{A_{ij}} \dbinom{\Xi_{ij}}{a}\dbinom{2\Xi_{ij}-\Xi_{ij}}{A_{ij}-a} = \dbinom{2\Xi_{ij}}{A_{ij}},
		\label{qe:Vandermonde}
	\end{equation}
	and from the fact that \(\dir{A}_{ii} = A_{ii}/2\), \(\forall i \in V\), it follows that \cref{eq:sum-dir:fold} is equivalent to \cref{eq:hyper-undir}.
\end{proof}

In the two sections above we have provided a parsimonious formulation of the soft-configuration model in terms of an \emph{hypergeometric ensemble}.
The ensemble provides a random graph model formulation for directed and undirected graphs alike, in which (i) the expected in- and out-degree sequences are fixed, and (ii) multi-edges between these vertices with fixed expected degrees are formed uniformly at random.
More precisely, the probability for a particular pair of vertices to be connected by an edge is only influenced by combinatorial effects, and thus only depends on the degrees of the vertices and the total number of edges.

\subsection{Generalised Hypergeometric Ensemble of Graphs}
As described above, the hypergeometric ensemble of random graphs samples edges uniformly at random from the urn containing all possible edges.
However, such uniform sampling of edges is generally not enough to describe real graphs.
In fact, in empirical graphs edge probabilities do not only depend on the activity of vertices (represented by their degrees) but also on other characteristics. Examples of such characteristics are vertex labels that lead to observable group structures, distances between vertices, etc.
Below we introduce how such influences on the sampling probabilities of edges can be encoded in a random graph model by means of a dyadic property we call \emph{edge propensity}.

First, we introduce the concept of edge propensity.
Given two dyads \((i,j)\) and \((k,l)\in V\times V\) where \(\Xi_{ij}=\Xi_{kl}\), according to the soft-configuration model of \cref{sec:conf} the probabilities of sampling one multi-edge between \((i,j)\) and \((k,l)\) are equal, leading to odds-ratio of \(1\) between the two pairs of vertices.
Instead, we generalise the model in a way that fixing arbitrary odds-ratios is possible.
That is, we define the aforementioned edge propensities such that the ratio between them is the odds-ratio of sampling one multi-edge between the two corresponding vertex pairs, all else being equal.
We use edge propensities to bias the sampling probability of each multi-edge, as illustrated in \cref{fig:genconf}.
This way, the probability of the number of multi-edges between a pair of vertices depends on both the degrees of the vertices and their edge propensity.
\begin{figure}
	\centering
	\includegraphics[width=.85\textwidth]{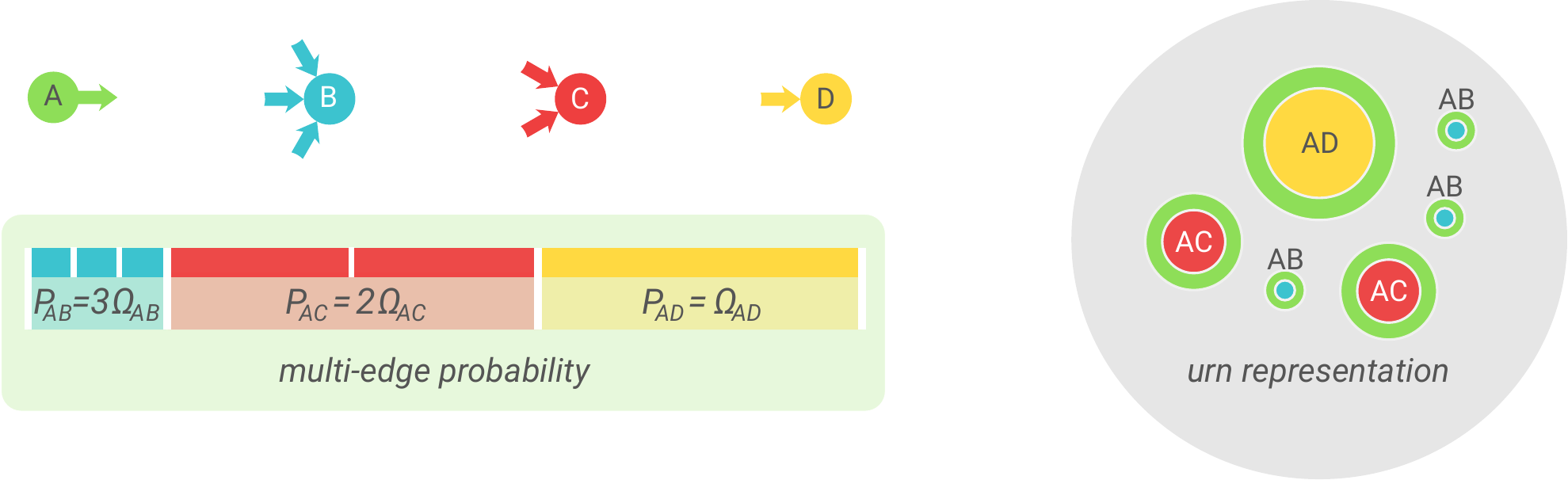}
	\caption{The effect of \emph{edge propensities} on the configuration model. Differently from the standard configuration model, here the stubs are not sampled uniformly at random (cf.~\cref{fig:softconf}).
	Given an out-stub, each in-stub is characterised by a propensity \(\Omega_{ij}\) of being chosen. As a result, the probability of wiring the out-stub \((A,\cdot)\) to the vertex \(D\) is larger than that of \(B\) due to a very large edge propensity \(\Omega_{AD}\), even though vertex \(B\) has three times more in-stubs than vertex \(D\).
	}\label{fig:genconf}
\end{figure}

We encode edge propensities in a matrix \(\mathbf{\Omega}\) defined as follows.
\begin{defn}[Propensity matrix]
	Let \(\mathbf{\Omega}=(\Omega_{ij})_{i,j\in V}\in\R^{\abs{V}\times\abs{V}}\) be a \(n\times n\) matrix where \(n\) is the number of vertices.
	Let \((i,j)\) and \((k,l)\in V\times V\).
	Let \(\omega_{ij,kl}\) the odds-ratio of sampling one multi-edge between \((i,j)\) instead of \((k,l)\).
	The entries \(\Omega_{ij}\) and \(\Omega_{kl}\) of the propensity matrix \(\mathbf{\Omega}\) are then defined such that \(\Omega_{ij}/\Omega_{kl}=\omega_{ij,kl}\).
	This implies that the propensity matrix \(\mathbf{\Omega}\) is defined up to a constant, as multiplying \(\mathbf{\Omega}\) with any constant preserves the specified odds-ratios.
\end{defn}

Now, we define a random graph model that combines the degree-related combinatorial effects, i.e., the configuration model, and the newly introduced edge propensities.
We do so by using the propensities to \emph{bias} the sampling process described in \cref{sec:conf}.
In the urn model analogy, such biased sampling implies that the probability of drawing a certain number of balls of a given colour (i.e., multi-edges between the corresponding pair of vertices) depends both  on their number and their size, as illustrated in \cref{fig:genconf}.
The probability distribution resulting from such a biased sampling process is given by the multivariate \emph{Wallenius' non-central hypergeometric distribution}~\cite{Wallenius1963,Fog2008a}.

\begin{thm}\label{thm:ghype}
	Let \(\mathcal G(V,E)\) be a directed graph with \(n=\abs{V}\) vertices and \(m=\abs{E}\) edges.
	Under the assumptions introduced above, the generalised hypergeometric ensemble of graphs (GHypEG) \(X\) induced by \(\mathcal G\) and a given propensity matrix \(\mathbf\Omega\) follows the \emph{multivariate Wallenius' non-central hypergeometric distribution} given in \cref{eq:walleniusNet}.
	
	Let \(\mathbf A\in \N^n\times\N^n\) be the adjacency matrix associated with \(\G\) and \(\mathbf\Xi\in \N^n\times\N^n\) be its \emph{combinatorial matrix} defined in \cref{eq:xi}.
	Then the GHypEG defined by \(\mathbf\Xi\) and \(\mathbf\Omega\), \(\G\) is distributed as follows:
	\begin{equation}
		\label{eq:walleniusNet}
		\Pr(X=\mathbf{A})=\left[\prod_{i,j\in V}{\dbinom{\Xi_{ij}}{A_{ij}}}\right]
		\int_{0}^{1}{\prod_{i,j\in V}{\left(1-z^{\frac{\Omega_{ij}}{S_{\mathbf{\Omega}} }}\right)^{A_{ij}}}dz}
	\end{equation}
	with
	\begin{equation}
		S_{\mathbf{\Omega}}= \sum_{i,j\in V} \Omega_{ij}(\Xi_{ij}-A_{ij}).
	\end{equation}
	
\end{thm}
The distribution describing the biased sampling from an urn is a generalisation of the multivariate hypergeometric distribution.
The proof of \cref{thm:ghype} follows from the fact that, when the sampling is performed without replacement with given relative odds, this sampling process corresponds to the multivariate Wallenius' non-central hypergeometric distribution.
Details of this derivation can be found in~\cite{Wallenius1963} for the univariate case, and in~\cite{Chesson1978,chesson1976non} for the multivariate case. 
A thorough review of non-central hypergeometric distributions has been done in~\cite{Fog2008}.

The next two corollaries directly follow from properties of Wallenius' non-central hypergeometric distribution.

\begin{cor}\label{cor:expwall}
	For each pair of vertices \(i,j \in V\), the probability to draw exactly \(A_{ij}\) edges between \(i\) and \(j\) is given by the marginal distributions of the multivariate Wallenius' non-central hypergeometric distribution, i.e.,
	\begin{equation}
		\label{eq:walleniusEdge}
		\begin{aligned}
			\Pr(&X_{ij}=A_{ij}) =  \dbinom{\Xi_{ij}}{A_{ij}}\dbinom{M-\Xi_{ij}}{m-A_{ij}}\cdot  \int_{0}^{1} \left[ \vphantom{e^\frac12}\right.
			&\left(
			1 - z^{ \frac{\Omega_{ij}}{S_{\mathbf{\Omega}}}}
			\right)^{A_{ij}} \left(
			1-z^{ \frac{\bar{\Omega}_{ij}}{S_{\mathbf{\Omega}}}}
			\right)^{m-A_{ij}}
			\left.\vphantom{\vphantom{e^\frac12}}\right]dz
		\end{aligned}
	\end{equation}
	where
	\begin{equation}	
		\bar{\Omega}_{ij} = \frac{\sum_{(l,m)\in
		\left(V\times V\right)\backslash(i,j)}{\Xi_{lm}\Omega_{lm}}}{(M-\Xi_{ij})}.
	\end{equation}
	
\end{cor}

\begin{cor}
	The entries of the expected adjacency matrix \(\expected{ X_{ij}}\) can be obtained by solving the following system of equations:
	\begin{align}\label{eq:walleniusMean}
	\left(1-\frac{\expected{ X_{11}}}{\Xi_{11}}\right)^{\frac{1}{\Omega_{11}}} = \left(1-\frac{\expected{ X_{12}}}{\Xi_{12}}\right)^{\frac{1}{\Omega_{12}}} = \ldots \end{align}
	with the constraint \(\sum_{i,j \in V} \expected{ X_{ij}} = m\).
\end{cor}

The undirected formulation of the GHypEG follows the same reasoning of \cref{thm:undirected}, with the addition of a symmetric propensity matrix.
That is, the distribution of the undirected generalised hypergeometric ensemble is hence given by a Wallenius' distribution similar to \cref{eq:walleniusNet}, but corresponding to the upper triangular part of the matrices (i.e., for \(i\leq j\)) with \(\mathbf{\Xi}\) defined according to \cref{lem:xiundir}.

\subsection{Estimation of the propensity matrix}
In this final section, we show how to define the propensity matrix \(\mathbf\Omega\) such that the expected graph \(\expected{X}\) from the model \(X\) defined by \(\mathbf\Omega\) coincides with an arbitrary graph \(G\).
By doing so, we create a random model \emph{centered} around the inducing graph.
The result is described in the following corollary, which follows from the properties of Wallenius' non-central hypergeometric distribution.

\begin{cor}
	Let \(\G(V,E)\) be a graph with \(n=\abs{V}\) vertices and \(\mathbf{A}\) its adjacency matrix.
	Let \(\mathbf\Omega\) be a \(n\times n\) propensity matrix, characterised by elements \(\Omega_{ij}\) with \(i,j\in V\).
	Then, \(\G\) coincides with the expectation \(\expected{ X}\) of the GHypEG \(X\) induced by \(\G\) and \(\mathbf\Omega\) if and only if the following relation holds.
	\begin{equation}\label{eq:fitomega}
		\forall c\in\R^-\quad\Omega_{ij} = \frac{1}{c}\log\left(1 - A_{ij} / \Xi_{ij}\right)\;\forall i,j\in V,
	\end{equation}
	Where \(\mathbf\Xi\) is the combinatorial matrix associated with \(X\) and \(\Xi_{ij}\) its elements.
\end{cor}
\begin{proof}
	\Cref{eq:fitomega} follows directly from \cref{cor:expwall}.
	In particular, when solving \cref{eq:walleniusMean} for \(\mathbf\Omega\) with the assumption \(\expected{X} = \mathbf{A}\) we obtain the following system of \(\abs{V}^2\) equations (in the case of a directed graph with self-loops)  for \(\abs{V}^2+1\) variables.
	\begin{equation}
		\begin{cases}
			\left(1-\frac{\expected{ X_{11}}}{\Xi_{11}}\right)^{\frac{1}{\Omega_{11}}} &= C \\
			\left(1-\frac{\expected{ X_{12}}}{\Xi_{12}}\right)^{\frac{1}{\Omega_{12}}} &= C \\
			&\vdots
		\end{cases}
	\end{equation}
	The solution of this system is \cref{eq:fitomega}.
\end{proof}

A wide range of statistical patterns that go beyond degree effects can be encoded in the graph model by specifying the matrix \(\mathbf{\Omega}\) of edge propensities. 
The encoding and fitting techniques of such arbitrary propensity matrices are beyond the scope of this article, and will not be discussed here.
We refer to~\cite{Casiraghi2017} for a general method to fit external dyadic data as propensities.

Finally, we show that the soft configuration model of \cref{sec:conf} is a special case of generalised hypergeometric graph models.
The soft configuration model described in \cref{thm:softconf} can be in fact recovered from the generalised model by setting all entries in the propensity matrix to the same value.
By doing so, the odds-ratio between the propensities for any pair of vertices is 1, and the edge sampling process is not biased.
Thus, the probability distribution of the model reduces to a function of the degree sequences and the number of multi-edges sampled.

\begin{thm}
	Let \(\mathbf{\Omega} \equiv \text{const}\).
	The corresponding GHypEG coincides with the soft configuration model in \cref{eq:softconf} induced by the same graph.
\end{thm}
\begin{proof}
	For the special case of a uniform edge propensity matrix \(\mathbf{\Omega} \equiv \text{const}\), which corresponds to an unbiased sampling of edges, for the integral in \cref{eq:walleniusEdge} we have
	\begin{equation}
		\int_{0}^{1}{\left(1-z^{\frac{1}{(M-m)}}\right)^m dz} = \dbinom{M}{m}^{-1}.
		\label{eq:walleniusIntegral}
	\end{equation}
	Plugging this result in \cref{eq:walleniusNet} we thus recover \cref{eq:hyper} for the unbiased case, i.e. where all edge propensities are identical.	
\end{proof}

\section{Final remarks}

We have proposed a novel approach for studying random graphs in terms of an urn problem.
By doing so, we have arrived at an analytically tractable formulation of the widely used configuration model of random graphs.
Furthermore, we have expanded the configuration model to the \emph{generalised hypergeometric ensemble}, which is a whole new class of models that can incorporate arbitrary dyadic biases in the probabilities of sampling edges.
Importantly, the analytical tractability of the generalised hypergeometric ensembles allows for robust model selection and hypothesis testing of various topological patterns in empirical data~\cite{Casiraghi2016}.
Moreover, one can perform a \emph{multiplex network regression}~\cite{Casiraghi2017} based on our model, in order to find the weighted combination of multiple relational layers that best describes the observed multi-graph.
For instance, one can study how different social phenomena and environmental factors influence the topological patterns in repeated social interactions.
The ability to analytically perform such statistical analysis opens new possibilities for the study of complex interaction data, which pervades many scientific disciplines.

\section*{Aknowledgements}
The authors thank Frank Schweitzer for his support and valuable comments, and Ingo Scholtes, Pavlin Mavrodiev, and Christian Zingg for the useful discussions.

\end{document}